\documentclass[twoside,10pt]{amsart}

\usepackage{amsthm,amsmath,amsfonts,epsfig,amssymb}
\usepackage[T1]{fontenc}
\usepackage{graphicx}
\usepackage{enumerate}
\usepackage{xy}
\usepackage{graphics}
\xyoption{all}
\entrymodifiers={+!!<0pt,\fontdimen22\textfont2>}

\newtheorem{thm}{Theorem}[section]

\newtheorem{prop} [thm]{Proposition}
\newtheorem{lem} [thm]{Lemma}
\newtheorem{cor}[thm]{Corollary}

\newtheorem*{thm*}{Theorem}
\theoremstyle{definition}
\newtheorem{defin}[thm]{Definition}
\theoremstyle{remark}
\newtheorem{rem}[thm]{Remark}

\def\C{\mathbb{C}}
\def\R{\mathbb{R}}

\def\Z{\mathbb{Z}}
\def\N{\mathbb{N}}

\def\A{\mathbb {A}}
\def\O{\mathcal O}

\def\exten#1#2#3#4{\mathrm{Ext}_{#2}^{#1}(#3,#4)}
\def\spec#1{\mathrm{Spec}(#1)}
\def\homm#1#2#3{\mathrm{Hom}_{#1}(#2,#3)}
\def\uni#1#2{(#2_1,\ldots,#2_{#1})}
\def\dim#1{\mathrm{dim}(#1)}

\def\K#1{\widetilde K_0Sp(#1)}
\def\Ko#1#2{\widetilde K_0Sp(#1,#2)}
\def\Sp#1{Sp_4(#1)}

\def\coker#1{\mathrm{coker}(#1)}
\def\ker#1{\mathrm{ker}(#1)}

\begin{document}

\title{Projective modules over the real algebraic sphere of dimension $3$}
\author{J. Fasel}
\date{}
 \address{Jean Fasel \\
EPFL SB SMA-GE \\
MA C3 595 (B\^atiment MA)\\
Station 8\\
CH-1015 Lausanne
 }
\email{jean.fasel@gmail.com}


\begin{abstract}
Let  $A$ be a commutative noetherian ring of Krull dimension $3$. We give a necessary and sufficient condition for $A$-projective modules of rank $2$ to be free. Using this, we show that all the finitely generated projective modules over the algebraic real $3$-sphere are free.
\end{abstract}

\maketitle
\tableofcontents


\section{Introduction}

Let $A$ be a noetherian commutative ring of Krull dimension $d$ and let $P$ be a projective $A$-module of rank $r$. Finding "computable" necessary and sufficient conditions for $P$ to be free is in general a very hard question. The first obvious obstruction is the class of $P$ in $\widetilde K_0(A)$ and, if $r>d$ its vanishing is also a sufficient condition by Bass-Schanuel cancellation theorem. Thus one is reduced to the case of stably free modules of rank $r\leq d$. In the critical case $r=d$, the situation is rather well understood. Observe first that a stably free module of rank $r=d$ can be seen as the kernel of a surjection $A^{d+1}\to A$ and that such surjections are given by rows $\uni {d+1}a$ such that the $a_i$'s generate $A$. Such a row is called unimodular, and it is easily seen that there is an action of $GL_{d+1}(A)$ on the set $Um_{d+1}(A)$ of unimodular rows of rank $d+1$. A stably free module is free if and only if its associated unimodular row is in the $GL_{d+1}(A)$ orbit of $(1,0,\ldots,0)$. This obstruction is "computable" in the following sense:
If $E_{d+1}(A)$ is the subgroup of $GL_{d+1}(A)$ generated by the elementary matrices, then there is a structure of abelian group on $Um_{d+1}(A)/E_{d+1}(A)$ (\cite[Theorem 4.1]{vdk}). If moreover $A$ is a smooth $k$-algebra of dimension $d\geq 3$ over a perfect field of characteristic different from $2$, this group can be seen as the cohomology group $H^d(A,G^{d+1})$, where $G^{d+1}$ is a suitable sheaf on $\spec A$ (\cite[Theorem 4.9]{Fa4}). The induced action of $SL_{d+1}(A)$ on $Um_{d+1}(A)/E_{d+1}$ reads in this particular situation as a group homomorphism $SL_{d+1}(A)/E_{d+1}(A)\to Um_{d+1}(A)/E_{d+1}(A)$ defined by sending a matrix to its first row. The cokernel of this map is in bijection with $Um_{d+1}(A)/SL_{d+1}(A)$, which is therefore computable in some situations (see \cite[Theorem 5.6]{Fa4} for some cases). 

If the projective module $P$ is of rank $r<d$, then the situation is much harder. For example, the question whether there exists a complex smooth algebra $A$ of dimension $3$ and a stably free non free projective $A$-module of rank $2$ has been solved only recently (\cite[Theorem 5.4]{Fa2}). In this paper, we prove the following theorem (Theorem \ref{cancellation} in the text):

\begin{thm*}
Let $A$ be a ring of dimension $3$. Then every projective module of rank $2$ with trivial determinant is free if and only if the group $\K A$ and the set $Um_4(A)/\Sp A$ are both zero.
\end{thm*}

The group $\K A$ appearing in the theorem is just the Grothendieck-Witt group $GW^2(A)$ modulo the subgroup generated by the antisymmetric form $H(A)$ and $Um_4(A)/\Sp A$ is just the set of orbits of $Um_4(A)$ under the action of $\Sp A$. The first group is "computable" in the sense that it is part of a cohomology theory.

As an application of the above result, we get the following theorem (Theorem \ref{main} in the text):

\begin{thm*}
Every finitely generated projective module over the real algebraic sphere of dimension $3$ is free.
\end{thm*}

It should be pointed out that the stable structure of the algebraic vector bundles on spheres is well-known (\cite{Fo}; \cite{Sw2}), but that the set of isomorphism classes of such vector bundles is much more difficult to compute. At the moment, this is known only for the circle $S^1$ (obvious) and the two-dimensional sphere $S^2$ (\cite{BO}). The result on $S^3$ could be thought of as a measure of the recent progresses made in understanding the unstable structure of projective modules. 

The paper is organized as follows: In Section \ref{symplectic}, we recall a few basic facts about symplectic modules. This includes the definitions of the group $\K A$, and the definition of the Euler class of a projective module of rank $2$. The following section is devoted to the study of the set of unimodular rows of even length $n$ under the action of the symplectic group $Sp_n(A)$. We show that a unimodular row of even length $n$ gives a stably free module of rank $n-2$ endowed with an antisymmetric form. This form is stably isometric to $H(A^{n/2})$ (whose class in $\K A$ is therefore trivial) and becomes isometric to $H(A^{(n-2)/2})$ if and only if its associated unimodular row is in the $Sp_{n}(A)$-orbit of $(1,0,\ldots,0)$. We prove the theorem on the freeness of projective modules of rank $2$ in Section \ref{core}. The basic (and classical) observation here is that every projective module of rank $2$ is endowed with a non-degenerated antisymmetric form. In Section \ref{gwss}, we recall the definition and the basic properties of the Gersten-Witt complex.  The last section shows that the group $\K A$ and the set $Um_4(A)/Sp_4(A)$ involved in the first theorem vanish when $A=S^3$ is the coordinate ring of the real algebraic sphere of dimension $d$. Since $\K {S^3}$ coincide with the derived Witt group $W^2(S^3)$ as defined by Balmer (\cite{BalSurvey}), we compute the latter. It must be pointed out that C. Walter announced the computation of the derived Witt groups of quadrics, and therefore we only compute the groups needed in our main result without going to much into the details. Our computation is rather technical, and probably one can find a better approach. We conclude with the proof of our theorem that all the finitely generated projective modules on $S^3$ are free.

\subsection{Conventions}
All the rings considered are commutative and noetherian. The dimension of a ring will always be its Krull dimension.
If $X$ is a scheme and $x\in X$, we denote by $\kappa(x)$ the residue field of $x$. If $X$ is integral, we simply denote by $\kappa(X)$ the residue field at the generic point.


\section{Symplectic modules}\label{symplectic}

\subsection{The group $\Ko AL$}
Let $A$ be a commutative ring and let $L$ be a line bundle over $A$. Let $GW^2(A,L)$ be the Grothendieck-Witt group of the (exact) category of projective $A$-modules endowed with the duality $\homm A{\_}L$ and the usual canonical isomorphism $\omega:Id\to \homm A{\homm A{\_}L}L$. Explicitly, $GW^2(A,L)$ is the Grothendieck group of the monoid (with the orthogonal sum $\bot$ as operation) of isometry classes of pairs $(P,\phi)$, where $P$ is a projective $A$-module and $\phi:P\to \homm APL$ is an anti-symmetric isomorphism.

If $Q$ is a projective $A$-module, we define $H(Q)$ to be the pair $(Q\oplus \homm AQL,h)$, where 
$$h:Q\oplus \homm APL\to \homm AQL\oplus \homm A{\homm APL}L$$ 
is defined by $h(q,f)(q^\prime,f^\prime)=f(q^\prime)-f^\prime(q)$. This induces a homomorphism 
$$H:K_0(A)\to GW^2(A,L)$$ 
whose cokernel is the Witt group $W^2(A,L)$.

We denote by $\Ko AL$ the quotient of $GW^2(A,L)$ by the subgroup generated by $H(A)$. If $L=A$, we simply denote by $\K A$ the group $\Ko AA$.

\subsection{Euler classes}

Let $P$ be a projective module of rank $2$. Then $P$ is endowed with an anti-symmetric form $\chi:P\to \homm AP{\det P}$ defined by $\chi(p)(q)=p\wedge q$. We denote by $e(P)$ the class of $(P,\chi)$ in $\Ko A{\det P}$. Suppose that $f:P\to Q$ is an isomorphism of (rank $2$) projective modules. Then $\det f$ induces an isomorphism $\Ko A{\det P}\to \Ko A{\det Q}$ under which $e(P)$ is sent to $e(Q)$. The Euler class satisfies good functorial properties, as shown in \cite[\S 2.4, \S 2.5]{FS}. In particular,

\begin{prop}\label{vanishing}
Suppose that $P\simeq (\det P)\oplus A$. Then $e(P)=0$.
\end{prop}

If $A$ is of dimension $2$, then the proposition is the easy part of a stronger result (\cite[Proposition 23]{FS}):

\begin{thm}
Let $A$ be a ring of dimension $2$. Then $e(P)=0$ if and only if $P\simeq (\det P)\oplus A$. 
\end{thm}

If $A$ is of dimension $3$, then the above result is no longer true, and even for a stably free module of rank $2$ we need another obstruction to detect whether it is free or not. This is the object of the next section.


\section{Unimodular rows under the symplectic group}\label{unimodular}

\subsection{Unimodular rows}

A unimodular row of length $n$ is a surjective homomorphism $A^n\to A$. Equivalently, a unimodular row is represented by a set $\uni na$ of elements of $A$ which generate the ideal $A$. The group $GL_n(A)$ acts on $Um_n(A)$ by composition, and therefore any subgroup of $GL_n(A)$ also acts. Traditionally, one is interested either in $E_n(A)$ (the subgroup generated by the elementary matrices) or in $GL_n(A)$ itself. Indeed, the set of orbits $Um_n(A)/GL_n(A)$ detects the stably free non free modules of rank $n-1$ as follows:

Let $v=\uni na$ be a unimodular row and let $P(v)$ be the projective module defined by the exact sequence
$$\xymatrix@C=2em{0\ar[r] & P(v)\ar[r] & A^n\ar[r]^-{v} & A\ar[r] & 0.}$$
It is easily seen that the following proposition holds:

\begin{prop}\label{free}
The modules $P(v)$ and $P(v^\prime)$ are isomorphic if and only if there is a matrix $G\in GL_n(A)$ such that $vG=v^\prime$.
\end{prop}
This proposition also shows that the set $Um_n(A)/GL_n(A)$ is quite hard to compute in general. If one restricts to $E_n(A)$, the situation is slightly better, as shown by the next result due to W. van der Kallen (\cite[Theorem 4.1]{vdk}):

\begin{thm}
Let $A$ be a ring of dimension $d$. If $n\geq (d+4)/2$ then $Um_n(A)/E_n(A)$ has the structure of an abelian group. 
\end{thm}

Of course, this doesn't imply that $Um_n(A)/E_n(A)$ becomes suddenly easy to compute. However, in the special situation when $A$ is a smooth algebra of dimension $d$ over a field of characteristic different from $2$, it is shown in \cite{Fa4} that $Um_{d+1}(A)/E_{d+1}(A)$ (same $d$) has a cohomological interpretation ([loc. cit., Theorem 4.9]). This led to the computation of those groups for any smooth rational oriented real algebra ([loc. cit.,Theorem 5.6]).  

Suppose now that $n$ is even. Then $A^n$ is the underlying module of the antisymmetric form $h:A^n\simeq (A^n)^\vee$ which is $\bot_{n/2}H(A)$. Any unimodular row $v$ of length $n$ gives a stably free module of rank $n-2$ endowed with an antisymmetric form as follows:

The exact sequence 
$$\xymatrix{0\ar[r] & P(v)\ar[r]^-i & A^n\ar[r]^-v & A\ar[r] & 0}$$
yields a commutative diagram
$$\xymatrix{0\ar[r] & A^\vee\ar[r]^-{v^\vee}\ar@{-->}[d]_-s & (A^n)^\vee\ar[r]^-{i^\vee}\ar[d]_-{h^{-1}} & P(v)^\vee\ar[r]\ar@{-->}[d]_-{-s^\vee} & 0 \\
0\ar[r] & P(v)\ar[r]_-i & A^n\ar[r]_-v & A\ar[r] & 0}$$ 
Since $h$ is antisymmetric $vh^{-1}v^\vee=0$ and therefore we get a homomorphism $s$ making the diagram commutative. Observe that $\ker {s^\vee}$ is projective and equal to $\coker s^\vee$. The snake lemma gives an isomorphism $\coker s\to \ker s^\vee$ which is antisymmetric.

\begin{defin}
If $v$ is a unimodular row of even length $n$, we denote by $(Q(v),\phi(v))$ the antisymmetric pair obtained above.
\end{defin}

\begin{lem}\label{isometry}
There is an isometry $(Q(v),\phi(v))\bot H(A)\simeq  H(A^{n/2})$. 
\end{lem}

\begin{proof}
Consider the following diagram
$$\xymatrix{ & & & Q(v)^\vee\ar[d]^-{j^\vee} &  \\
0\ar[r] & A\ar[r]^-{v^\vee}\ar[d]_-s & (A^n)^\vee\ar[r]^-{i^\vee}\ar[d]_-{h^{-1}}\ar@{-->}[ru]^-{\gamma^\vee} & P(v)^\vee\ar[r]\ar[d]_-{-s^\vee} & 0\\
0\ar[r] & P(v)\ar[r]_-i\ar[d]_-j & A^n\ar[r]_-v & A\ar[r] & 0  \\
& Q(v)\ar@{-->}[ru]_-\gamma & &  &}$$
Choosing an element $w\in A^n$ such that $v\cdot w=1$ we get a section $\alpha$ of $v$ and therefore a section $r=i^\vee h\alpha$ of $-s^\vee$ and a retraction $-r^\vee$ of $s$. This gives a section $\beta$ of $j$ and we get an injective homomorphism $\gamma:Q(v)\to A^n$. One checks easily that $\phi(v)=\gamma^\vee h\gamma$ and that we have $A^n=Ah^{-1}v^\vee\oplus Aw\oplus \gamma Q(v)$. Moreover, $w$ and $h^{-1}v^\vee$ are both in $\gamma Q(v)^\bot$ and they generate a submodule isometric to $H(A)$.
\end{proof}

The following result is the analogue of Proposition \ref{free} for "decorated" modules:

\begin{prop}\label{sympl}
The pairs $(Q(v),\phi(v))$ and $(Q(v^\prime),\phi(v^\prime))$ are isometric if and only if there is a matrix $M\in Sp_n(A)$ such that $vM=v^\prime$.
\end{prop}

\begin{proof}
Following the construction of $Q(v^\prime)$ for $v^\prime=vM$, we see that in that case $(Q(v^\prime),\phi(v^\prime))$ is isometric to $(Q(v),\phi(v))$. Lemma \ref{isometry} gives the converse statement.
\end{proof}

This shows that $Um_n(A)/Sp_n(A)$ classifies the stably free modules of rank $n-2$ decorated with the antisymmetric form obtained above.

\section{Cancellation of symplectic modules}\label{core}

Let $A$ be a ring of odd dimension $d\geq 3$. Let $(Q,\phi)$ be an antisymmetric pair, where $Q$ is of rank $d-1$. Suppose that there is an isometry 
$$f:H(A)\bot (Q,\phi)\simeq H(A^{(d+1)/2})$$
The projection of $A^2$ to the first factor gives a homomorphism $e_1^*:A^2\oplus Q\to A$. Then $e_1^*f^{-1}$ is a unimodular row of length $d+1$. If $g$ is another such isometry, then $gf^{-1}\in Sp_{d+1}(A)$ and $e_1^*g^{-1}(gf^{-1})=e_1^*f^{-1}$. Therefore, the class of $e_1^*f^{-1}$ in $Um_{d+1}(A)/Sp_{d+1}(A)$ is independent of the isometry $f$.

\begin{defin}
If $(Q,\phi)$ is a symmetric pair such that $Q$ is of rank $d-1$ and $H(A)\bot (Q,\phi)\simeq H(A^{(d+1)/2})$, we denote by $sp(Q,\phi)$ the class of $e_1^*f^{-1}$ in the set $Um_{d+1}(A)/Sp_{d+1}(A)$ for any isometry $f$. We call it the \emph{symplectic class} of $(Q,\phi)$.
\end{defin}

This definition seems rather artificial at first sight. However, the next proposition shows that it is natural.

\begin{prop}\label{identification}
Let $v:=sp(Q,\phi)$. Then $(Q,\phi)$ is isometric to $(Q(v),\phi(v))$.
\end{prop}

\begin{proof}
Let $f:H(A)\bot (Q,\phi)\simeq H(A^{(d+1)/2})$ be the isometry defining $sp(Q,\phi)$.
Consider the commutative diagram
$$\xymatrix{0\ar[r] & Q\oplus A\ar[r]\ar@{-->}[d]_-{f^\prime} & Q\oplus A^2\ar[r]^-{e_1^*}\ar[d]_-f & A\ar[r]\ar@{=}[d] & 0\\
0\ar[r] & P(v)\ar[r] & A^n\ar[r]_-{v} & A\ar[r] & 0}$$
The construction of $(Q(v),\phi(v))$ shows that $f^\prime:Q\oplus A\to P(v)$ induces an isometry $(Q,\phi)\simeq (Q(v),\phi(v))$. 
\end{proof}

So we have a procedure to understand when an antisymmetric pair $(Q,\phi)$ is isometric to $H(A^{(d-1)/2})$. The first obstruction is its class in $\K A$. If this is different from $0$, then it cannot be isometric to $H(A^{(d-1)/2})$. If the class is zero, then (using \cite[Chapter IV, Corollary 4.15]{Bass2}) we get an isometry  
$$(Q,\phi)\bot H(A)\simeq (A^{(d+1)/2}).$$
and therefore a symplectic class in $Um_{d+1}(A)/Sp_{d+1}(A)$. Using Propositions \ref{identification} and \ref{sympl}, we see that this class must be trivial in order for $(Q,\phi)$ to be isometric to $H(A^{(d-1)/2})$. 

In dimension $3$, this procedure gives a complete understanding of projective modules of rank $2$.
\begin{thm}\label{cancellation}
Let $A$ be a ring of dimension $3$. Then every projective module of rank $2$ with trivial determinant is free if and only if the group $\K A$ and the set $Um_4(A)/\Sp A$ are both trivial.
\end{thm}

\begin{proof}
Remark first that every element in $\K A$ is of the form $e(P)$ for some projective module $P$ of rank $2$ (use \cite[Chapter IV, Theorem 4.12]{Bass2}). If every such module is free, then $\K A=0$ because the only non degenerate anti-symmetric form on $A^2$ is $h$ (up to isometry). The set $Um_4(A)/\Sp A$ is reduced to a point because of Proposition \ref{sympl}.

Suppose now that $\K A=0$. Then the Euler class of $P$ is zero for any $P$ of rank $2$, and its symplectic class (which is then defined) is in the $Sp_4(A)$-orbit of $(1,0,0,0)$. Therefore any such $P$ is free by Proposition \ref{sympl} again.
\end{proof}

In the last section, we show that the real algebraic sphere of dimension $3$ satisfy the hypothesis of the theorem. We first need some preparation in order to compute the group $\K {S^3}$. This is the object of the next section.


\section{The Gersten-Witt complex}\label{gwss}
\subsection{The complex}
Let $X$ be an integral regular scheme over a field $k$. Let $d=\dim X$. The \emph{Gersten-Witt complex} is the complex
$$\xymatrix@C=1.3em{0\ar[r] & W(\kappa(X))\ar[r] & \displaystyle{\bigoplus_{x_1\in X^{(1)}} W(\kappa(x_1))   }\ar[r] & \ldots  \ar[r] &  \displaystyle{\bigoplus_{x_d\in X^{(d)}} W(\kappa(x_d))   }\ar[r] & 0}$$
as constructed in \cite{BW}. Observe that the Witt groups in the complex are in fact twisted Witt groups, i.e. Witt groups with twisted duality $\homm {\kappa(x_p)}{\_}{\omega_{x_p}}$, where $\omega_{x_p}=\exten p{\O_{X,x_p}}{\kappa(x_p)}{\O_{X,x_p}}$ (which is a $\kappa(x_p)$-vector space of dimension $1$ because $X$ is regular). We denote by $H^i(X,W)$ the cohomology groups of this complex. If $d\leq 3$, we have that $H^i(X,W)=W^i(X)$ for $0\leq i\leq 3$ (\cite{BW} again). 

Even if the Witt groups in the complex are twisted by line bundles, we can still define the fundamental ideal and the groups behaving like the powers of it (just choose any isomorphism between the twisted Witt group and the classical Witt group, see for instance \cite[Definition 9.2.1]{Fa1}). It turns out that the differentials of the Gersten-Witt complex respect the powers of the fundamental ideal (\cite[Theorem 9.2.4]{Fa1}, or \cite{Gi2}), and therefore we get a complex
$$\xymatrix@C=1.3em{0\ar[r] & I^j(\kappa(X))\ar[r] & \displaystyle{\bigoplus_{x_1\in X^{(1)}} I^{j-1}(\kappa(x_1))   }\ar[r] & \ldots  \ar[r] &  \displaystyle{\bigoplus_{x_d\in X^{(d)}} I^{j-d}(\kappa(x_d))   }\ar[r] & 0}$$
for any $j$ (where $I^l(\kappa(x_p))$ is defined to be $W(\kappa(x_p))$ if $l\leq 0$). We denote by $H^i(X,I^j)$ the cohomology groups of this complex. We get $H^i(X,W)=H^i(X,I^j)$ when $j\leq 0$ by definition.

Let $\overline I^j(\kappa(x_p))=I^j(\kappa(x_p))/I^{j+1}(\kappa(x_p))$. Since the Gersten-Witt complex respects the powers of the fundamental ideal, we get a complex
$$\xymatrix@C=1.3em{0\ar[r] &\overline I^j(\kappa(X))\ar[r] & \displaystyle{\bigoplus_{x_1\in X^{(1)}}\overline I^{j-1}(\kappa(x_1))   }\ar[r] & \ldots  \ar[r] &  \displaystyle{\bigoplus_{x_d\in X^{(d)}}\overline I^{j-d}(\kappa(x_d))   }\ar[r] & 0}$$
which coincides with the Gersten complex in \'etale cohomology with coefficient in $\mu_2$ by Voevodsky's results (\cite{OVV} and \cite{Vo}). For any $i,j\in\Z$, let $H^i(X,\overline I^j)$ denote the cohomology groups of this complex. Observe that $H^i(X,\overline I^j)=0$ if $i>j$ and that $H^i(X,\overline I^i)=CH^i(X)/2$ (\cite[Theorem 9.1]{Gi2}). By construction, we have for any $j$ a long exact sequence in cohomology
$$\xymatrix@=1.3em{0\ar[r] & H^0(X,I^{j+1})\ar[r] & H^0(X,I^j)\ar[r] & H^0(X,\overline I^j)\ar[r] & H^1(X,I^{j+1})\ar[r] & \ldots }$$
This sequence (and the fact that $H^i(X,\overline I^j)=0$ if $i>j$) yields the following Lemma whose proof is obvious:

\begin{lem}\label{comparison}
Let $X$ be a regular scheme over a field $k$. Then for any $j\in\Z$ we have $H^i(X,W)=H^i(X,I^j)$ if $i\geq j+1$. Moreover, the natural homomorphism $H^j(X,I^j)\to H^j(X,W)$ is surjective.
\end{lem}

In case $j=1$, one can say a bit more:

\begin{lem}\label{deg_1}
Let $X$ be a regular scheme over a field $k$. Then $H^1(X,I)=H^1(X,W)$.
\end{lem}

\begin{proof}
Using the long exact sequence in cohomology 
$$\xymatrix{H^0(X,W)\ar[r] & H^0(X,\overline W)\ar[r] & H^1(X,I)\ar[r] & H^1(X,W)\ar[r] & 0,}$$
we see that it suffices to prove that the homomorphism $H^0(X,W)\to H^0(X,\overline W)$ is surjective to conclude. Since the Gersten-Witt complex of a scheme is the direct sum of the Gersten-Witt complexes of its connected components, we can suppose that $X$ is connected. Let $p:X\to k$ be the projection. Then we have a commutative diagram
$$\xymatrix{H^0(k,W)\ar[r]\ar[d]_-{p^*} & H^0(k,\overline W)\ar@{=}[d]\\
H^0(X,W)\ar[r] & H^0(X,\overline W)}$$
and the above horizontal homomorphism is clearly surjective.
\end{proof}

\subsection{Supports}
As in the previous section, let $X$ be an integral scheme of dimension $d$ over a field $k$. If $Y\subset X$ is a closed subscheme, one can consider the \emph{Gersten-Witt complex with support on $Y$}:
$$\xymatrix@C=0.9em{0\ar[r] & \displaystyle{\bigoplus_{x_0\in X^{(0)}\cap Y} W(\kappa(x_0))   }\ar[r] & \displaystyle{\bigoplus_{x_1\in X^{(1)}\cap Y} W(\kappa(x_1))   }\ar[r] & \ldots  \ar[r] &  \displaystyle{\bigoplus_{x_d\in X^{(d)}\cap Y} W(\kappa(x_d))   }\ar[r] & 0}$$
and its filtered version
$$\xymatrix@C=1.3em{0\ar[r] & \displaystyle{\bigoplus_{x_0\in X^{(0)}\cap Y} I^j(\kappa(x_0))   }\ar[r] & \ldots  \ar[r] &  \displaystyle{\bigoplus_{x_d\in X^{(d)}\cap Y} I^{j-d}(\kappa(x_d))   }\ar[r] & 0.}$$
We denote by $H^i_Y(X,I^j)$ the cohomology groups of this complex. In case $Y$ is the zero locus of a global section $f\in \O_X(X)$, we simply denote by $H^i_f(X,I^j)$ the group $H^i_Y(X,I^j)$. All the results of Section \ref{gwss} remains true for groups with support. In addition, there is a long exact sequence
$$\xymatrix@C=1.4em{\ldots\ar[r] & H^i_Y(X,I^j)\ar[r] & H^i(X,I^j)\ar[r] & H^i(U,I^j)\ar[r] & H^{i+1}_Y(X,I^j)\ar[r] & \ldots}$$
for any $j\in\Z$, where $U=X-Y$. If $Y\to X$ is smooth, one can compare the cohomology groups with support on $Y$ with the cohomology groups of $Y$ (\cite[Remark 9.3.5]{Fa1}). We will need only the following particular result (\cite[Theorem 4.1]{Gi1} and \cite[Remark 9.3.5]{Fa1}) :

\begin{lem}\label{reg_param}
Let $X=\spec A$ be a regular affine scheme over a field $k$. Let $f\in A$ be a regular element such that $Y=\spec {A/f}$ is regular. Then 
$$H^i_f(X,I^j)\simeq H^{i-1}(Y,I^{j-1})$$
for any $i,j\in\Z$.
\end{lem}


\section{The real sphere of dimension $3$}\label{witt}

Our goal of the section is to compute the group $\K {S^3}$, which coincide with the Witt group $W^2(S^3)$ (see Corollary \ref{s3}). As $S^3$ is of dimension $3$, $W^2(S^3)=H^2(S^3,W)$ which is really the group that we are going to compute. 

\subsection{Auxiliary computations}

For any $n\in\N$, let $S^n=\R[x_0,\ldots,x_n]/(\sum x_i^2-1)$. Using the long exact sequence of localization, we see that it suffices to compute the Witt groups of $S^{n-1}$ and $S^n_{x_n}$ to compute the Witt groups of $S^n$. The ring $S^n_{x_n}$ has a nice expression:

 If $B^n=\R[x_0,\ldots,x_n]/(\sum_{i=0}^{n-2} x_i^2-x_{n-1}x_n+1)$ (for $n\geq 1$), there is an isomorphism
$$\psi:S^n_{x_n}\to B^n_{x_{n-1}+x_n}$$
defined by 
$$\psi(x_i)=\left\{ \begin{array}{ll} \frac {2x_i}{(x_{n-1}+x_n)} & \mbox{for any $0\leq i\leq n-2$,}\\
 \frac {x_n-x_{n-1}}{(x_{n-1}+x_n)} & \mbox{if $i=n-1$,}\\
 \frac 2{(x_{n-1}+x_n)} & \mbox{if $i=n$.}
 \end{array}\right.$$
Its inverse is the homomorphism $\phi:B^n_{x_{n-1}+x_n}\to S^n_{x_n}$ defined by
$$\phi(x_i)=\left\{ \begin{array}{cl} \frac {x_i}{x_n} & \mbox{for any $0\leq i\leq n-2$,}\\
 \frac {1-x_{n-1}}{x_n} & \mbox{if $i=n-1$,}\\
 \frac {1+x_{n-1}}{x_n} & \mbox{if $i=n$.}
 \end{array}\right.$$
The computation of $H^i(B^n,W)$ with $n$ small is not very difficult. To achieve this, we need an additional auxiliary ring $C^n=\R[x_0,\ldots,x_n]/(\sum x_i^2+1)$.

\begin{lem}\label{comput_c1}
We have 
$$H^i(C^1,W)=\left\{ \begin{array}{cl} \Z/4 & \mbox{if $i=0$,}\\
0 & \mbox{if $i\geq 1$.}\end{array}\right.$$
\end{lem}

\begin{proof}
Since $C^1$ is a curve, $H^i(C^1,W)=0$ for $i\geq 2$. We also get $H^0(C^1,W)=W^0(C^1)=\Z/4$, generated by the class of $\langle 1\rangle$ by \cite[Theorem 2.9]{Mon}. To compute $H^1$, we first observe that $H^1(C^1,I)=H^1(C^1,W)$ by Lemma \ref{deg_1}. 

Using \cite[Theorem 2']{EL}, we get that $I^2(\kappa(C^1))=0$. Moreover, since $C^1$ has no real point then $I(\kappa(x_1))=0$ for any closed point $x_1$. Therefore, $H^1(C^1,I)\simeq H^1(C^1,\overline I)$. The latter is just $Pic(C)/2$, which is trivial by \cite{Sw1}.
\end{proof}

\begin{lem}\label{comput_c2}
We have 
$$H^i(C^2,W)=\left\{ \begin{array}{cl} 
\Z/2 & \mbox{if $i=1$,}\\
0 & \mbox{if $i\geq 2$.}
\end{array}\right.$$
\end{lem}

\begin{proof}
Since $C^2$ is a surface, we have $H^i(C^2,W)=0$ if $i\geq 3$. Because $C^2$ has no real point, we have $H^2(C^2,I^2)=0$ by \cite[Proposition 5.1]{Fa4} (observe that $C^2$ has a trivial canonical bundle). This shows that $H^2(C^2,W)=H^2(C^2,I)=0$.

Lemma \ref{deg_1} shows that $H^1(C^2,I)=H^1(C^2,W)$. We then use the sequence comparing $I$ and $I^2$:
$$\xymatrix@C=1.4em{\ldots\ar[r] & H^1(C^2,I^2)\ar[r] & H^1(C^2,I)\ar[r] & H^1(C^2,\overline I)\ar[r] & H^2(C^2,I^2)\ar[r] & H^2(C^2,I)\ar[r] & 0.}$$
which reduces to 
$$\xymatrix@C=1.4em{\ldots\ar[r] & H^1(C^2,I^2)\ar[r] & H^1(C^2,I)\ar[r] & H^1(C^2,\overline I)\ar[r]& 0}$$
by the above discussion.

Since none of the residue fields of $\spec {C^2}$ admits an ordering, we obtain $H^i(C^2,I^2)=H^i(C^2,\overline I^2)$ (argue as in \cite{Suj}). Now $H^1(C^2,\overline I^2)=0$ by \cite[Corollary 3.3, Remark 3.3.1]{CS}. This implies that the map $H^1(C^2,I)\to H^1(C^2,\overline I)$ is an isomorphism. Since $Pic(C^2)/2=\Z/2$ by \cite[Theorem 9.2(b)]{Sw1}, we get $H^1(C^2,W)=\Z/2$.
\end{proof}

\begin{rem}\label{explicit_c2}
The generator of $H^1(C^2,W)$ can be explicitly obtained. Using \cite[Remark, p 471]{Sw1}, we find that the generator of $Pic(C^2)$ is the class of the prime ideal $\mathfrak p_1=(x_0^2+1, x_0x_1+x_2)$. Since $(x_0^2+1)(x_1^2+1)=(x_0x_1+x_2)(x_0x_1-x_2)$ and $(x_0x_1-x_2)\not\in\mathfrak p_1$ we see that $\mathfrak p_1C^2_{\mathfrak p_1}$ is generated by $x_0^2+1$. We thus get an isomorphism (necessarily symmetric) of $\kappa(\mathfrak p_1)$-vector spaces
$$\rho_{x_0^2+1}:\kappa(\mathfrak p_1)\to \exten 1{C^2_{\mathfrak p_1}}{\kappa(\mathfrak p_1)}{C^2_{\mathfrak p_1}}$$
defined by $1\mapsto Kos(x_0^2+1)$, where the latter is the Koszul complex associated to the regular parameter $x_0^2+1$. 

Consider next $\mathfrak p_2=(x_2)$. As above, we get a symmetric isomorphism of $\kappa(\mathfrak p_2)$-vector spaces
$$\rho_{x_2}:\kappa(\mathfrak p_2)\to \exten 1{C^2_{\mathfrak p_2}}{\kappa(\mathfrak p_2)}{C^2_{\mathfrak p_2}}$$
defined by $1\mapsto Kos(x_2)$, where the latter is the Koszul complex associated to the regular parameter $x_2$. 

The isomorphism $H^1(C^2,I)\simeq H^1(C^2,\overline I)$ (and the fact that $\mathfrak p_2$ is principal) shows that the class of the cycle $(\kappa(\mathfrak p_1),\rho_{x_0^2+1})+(\kappa(\mathfrak p_2),x_1\rho_{x_2})$ is the generator of $H^1(C^2,I)$ (observe also that we could have considered $(\kappa(\mathfrak p_1),\rho_{x_0^2+1})$ alone as a generator, but we will need $(\kappa(\mathfrak p_2),\rho_{x_2})$ later and this is why we consider it).
\end{rem}

Observe next that $B^1\simeq \R[x_0,x_0^{-1}]$ and $B^n/x_n=C^{n-2}[x_{n-1}]$ for $n\geq 2$, while $B^n_{x_n}\simeq \R[x_0,\ldots,x_{n-2},x_n,x_n^{-1}]=B^1[x_0,\ldots,x_{n-2}]$. 

\begin{lem}\label{comput_bn}
For $1\leq n\leq 3$, we have 
$$H^i(B^n,W)=\left\{ \begin{array}{cl} \Z\oplus \Z & \mbox{if $i=0$,}\\
0 & \mbox{if $i\geq 1$.}
\end{array}\right.$$
\end{lem}

\begin{proof}
Since $B^1=\R[x_0,x_0^{-1}]$, the result for $B^1$ follows from \cite[Corollary 5.2]{Gi1}. Since $B^n_{x_n}\simeq B^1[x_0,\ldots,x_{n-2}]$, the exact sequence of localization (and the homotopy invariance of the Witt cohomology, see \cite[Theorem 11.2.9]{Fa1}) yields  
$$\xymatrix{0\ar[r] & H^0(B^n,W)\ar[r] & H^0(B^1,W)\ar[r] & H^1_{x_n}(B^n,W)\ar[r] & \ldots}$$
Using Lemma \ref{reg_param} and the computation of $H^i(B^1,W)$, we get an exact sequence
$$\xymatrix{0\ar[r] & H^0(B^n,W)\ar[r] & \Z\oplus \Z\ar[r] & H^0(C^{n-2},W)\ar[r] & H^1(B^n,W)\ar[r] & 0}$$
and isomorphisms $H^{i-1}(C^{n-2},W)\simeq H^i(B^n,W)$ for $i\geq 2$. Since $C^0\simeq \C$, we have $H^{i-1}(C^0,W)=0$ for $i\geq 2$ and Lemma \ref{comput_c1} shows that the same statement is true for $C^1$. Hence $H^i(B^n,W)=0$ for $i\geq 2$ and it suffices to prove that the map $\Z\oplus \Z\to H^0(C^{n-2},W)$ in the above exact sequence is surjective to conclude. The two generators of $H^0(B^1,W)\simeq \Z\oplus \Z$ are the class of $\langle 1\rangle$ and $\langle x_n\rangle$. Now it is not hard to see that $\langle x_n\rangle$ is sent to the generator of $H^0(C^{n-2},W)$ (which is either $\Z/2$ or $\Z/4$) under the map $\Z\oplus \Z\to H^0(C^{n-2},W)$. Whence the result.
\end{proof}

Next we compute part of the Witt cohomology of $B^3_{x_2+x_3}$. It is of course also possible to compute $H^0(B^3_{x_{2}+x_3},W)$, but the answer is not "nice" and we won't need it.  

\begin{lem}\label{comput_b3}
We have 
$$H^i(B^3_{x_2+x_3},W)=\left\{ \begin{array}{cl}
\Z/2 & \mbox{if $i=1$,}\\
0 & \mbox{if $i\geq 2$.}
\end{array}\right.$$
\end{lem}

\begin{proof}
Since we already know the Witt cohomology groups of $B^3$ (Lemma \ref{comput_bn}) and $B^3/(x_2+x_3)=C^2$ (Lemma \ref{comput_c2}), we only have to use the exact sequence of localization. The connecting homomorphism yields isomorphisms $H^i(B^3_{x_2+x_3},W)\simeq H^i(C^2)$ for $i=1,2,3$. This shows the assertion.
\end{proof}

\begin{rem}\label{explicit_b3}
We can give an explicit generator of $H^1(B^3_{x_2+x_3},W)$. Consider the prime ideals $\mathfrak q_1=(x_0^2+1)$ and $\mathfrak q_2=(x_2)$. Taking the Koszul complexes associated to these regular sequences as generators of respectively $\exten 1{B^3_{\mathfrak q_1}}{\kappa({\mathfrak q_1})}{B^3_{\mathfrak q_1}}$ and $\exten 1{B^3_{\mathfrak q_2}}{\kappa({\mathfrak q_2})}{B^3_{\mathfrak q_2}}$, we get symmetric isomorphisms
$$\rho_{x_0^2+1}:\kappa({\mathfrak q_1})\to \exten 1{B^3_{\mathfrak q_1}}{\kappa({\mathfrak q_1})}{B^3_{\mathfrak q_1}}$$
and 
$$\rho_{x_2}:\kappa({\mathfrak q_2})\to \exten 1{B^3_{\mathfrak q_2}}{\kappa({\mathfrak q_2})}{B^3_{\mathfrak q_2}}$$
as in Remark \ref{explicit_c2}. Consider the element $(\kappa({\mathfrak q_1}),(x_0x_1+x_2)\rho_{x_0^2+1})+(\kappa({\mathfrak q_2}), x_1x_3\rho_{x_2})$ in $\bigoplus_{x_1\in \spec {B^3}^{(1)}}W(\kappa(x_1))$. A straightforward computation shows that the element $(\kappa({\mathfrak q_2}), x_1x_3\rho_{x_2})$ is ramified exactly in the prime ideals $(x_1,x_2)$ and $(x_2,x_3)$ (which is the same as $(x_2,x_2+x_3)$). On the other hand, $(\kappa({\mathfrak q_1}),(x_0x_1+x_2)\rho_{x_0^2+1})$ is ramified in both $(x_1,x_2)$ and in the prime ideal $\mathfrak p_1$ of Remark \ref{explicit_c2}. Indeed, the equality in $B^3$
$$(x_0x_1+x_2)(x_0x_1-x_2)-(x_0^2+1)(x_1^2+1)+x_2(x_2+x_3)=0$$
shows that a prime ideal containing $(x_0x_1+x_2)$ and $x_0^2+1$ must contain either $x_2$, either $x_2+x_3$. Such an ideal must therefore contain $(x_1,x_2)$ in the first case, and $(x_2+x_3, x_0^2+1,x_0x_1+x_2)=\mathfrak p_1$ (seen as an ideal in $B^3$) in the second case.

To check that $(\kappa({\mathfrak q_1}),(x_0x_1+x_2)\rho_{x_0^2+1})+(\kappa({\mathfrak q_2}), x_1x_3\rho_{x_2})$ indeed defines a class in $H^1(B^3_{x_2+x_3},W)$, we must compute its image under the differential in the Gersten-Witt complex associated to $B^3_{x_2+x_3}$. The image of $(\kappa({\mathfrak q_1}),(x_0x_1+x_2)\rho_{x_0^2+1})$ is the localization in the minimal prime ideals of its support of the complete intersection module $M_1:=B^3_{x_2+x_3}/(x_0^2+1,x_0x_1+x_2)$ endowed with the symmetric isomorphism
$$\rho_{x_0^2+1,x_0x_1+x_2}:M_1\to \exten 2{B^3_{x_2+x_3}}{M_1}{B^3_{x_2+x_3}}$$
defined by $\rho_{x_0^2+1,x_0x_1+x_2}(1)=Kos(x_0^2+1,x_0x_1+x_2)$, where the latter is the Koszul complex given by the regular sequence $(x_0^2+1,x_0x_1+x_2)$. It follows from the discussion above that the only minimal prime ideal in the support of $M_1$ is $(x_1,x_2)$. To compute the localization of $M_1$ and $\rho_{x_0^2+1,x_0x_1+x_2}$ in the prime ideal $(x_1,x_2)$ we compare the regular sequences $(x_0^2+1,x_0x_1+x_2)$ and $(x_1,x_2)$. Since $x_0^2+1=x_2x_3-x_1^2$ in $B^3$, we get 
$$\begin{pmatrix} -x_1 & x_3 \\ x_0 & 1\end{pmatrix}\begin{pmatrix} x_1 \\ x_2\end{pmatrix}=\begin{pmatrix} x_0^2+1 \\ x_0x_1+x_2\end{pmatrix}.$$
The determinant of the matrix is $-x_1-x_0x_3$, which is not in the ideal $(x_1,x_2)$ and therefore the regular sequences $(x_0^2+1,x_0x_1+x_2)$ and  $(x_1,x_2)$ generate the same ideal in $B^3_{(x_1,x_2)}$ and $Kos(x_0^2+1,x_0x_1+x_2)=(-x_1-x_0x_3)Kos(x_1,x_2)$. Since the residue field at the prime ideal $(x_1,x_2)$ is $\C(x_3)$ and the class of $-x_1-x_0x_3$ in this field is $-ix_3$, we see that the image of $(\kappa({\mathfrak q_1}),(x_0x_1+x_2)\rho_{x_0^2+1})$ under the differential in the Gersten-Witt complex is the symmetric isomorphism
$$\C(x_3)\to \exten 2{B^3_{(x_1,x_2)}}{\C(x_3)}{B^3_{(x_1,x_2)}}$$
defined by $1\mapsto -ix_3Kos(x_1,x_2)$. Since $-i$ is a square, this is equivalent to the symmetric isomorphism
$$\C(x_3)\to \exten 2{B^3_{(x_1,x_2)}}{\C(x_3)}{B^3_{(x_1,x_2)}}$$
defined by $1\mapsto x_3Kos(x_1,x_2)$. It is easy to check that the image of $(\kappa({\mathfrak q_2}), x_1x_3\rho_{x_2})$ is the symmetric isomorphism
$$\C(x_3)\to \exten 2{B^3_{(x_1,x_2)}}{\C(x_3)}{B^3_{(x_1,x_2)}}$$
defined by $1\mapsto x_3Kos(x_2,x_1)$. Since $Kos(x_1,x_2)=-Kos(x_2,x_1)$, we see that $(\kappa({\mathfrak q_1}),(x_0x_1+x_2)\rho_{x_0^2+1})+(\kappa({\mathfrak q_2}), x_1x_3\rho_{x_2})$ indeed defines an element of the group $H^1(B^3_{x_2+x_3},W)$.

We next show that its image under the isomorphism 
$$H^1(B^3_{x_2+x_3},W)\to H^1(C^2,W)$$ 
is precisely the generator obtained in Remark \ref{explicit_c2}. We start with the computation of the image of $(\kappa({\mathfrak q_1}),(x_0x_1+x_2)\rho_{x_0^2+1})$. Arguing as above and using the equality 
$$(x_0x_1+x_2)(x_0x_1-x_2)=(x_0^2+1)(x_1^2+1)-x_2(x_2+x_3)$$
we see that $Kos(x_0^2+1,x_0x_1+x_2)=-x_2/(x_0x_1-x_2)Kos(x_0^2+1,x_2+x_3)$ in $\exten 2{B^3_{\mathfrak p_1}}{\kappa(\mathfrak p_1)}{B^3_{\mathfrak p_1}}$. Since $(x_0x_1-x_2)=-2x_2$ in $\kappa(\mathfrak p_1)$ (because $x_0x_1+x_2$ is in $\mathfrak p_1$) and $-1/2$ is a square, we see that the image of $(\kappa({\mathfrak q_1}),(x_0x_1+x_2)\rho_{x_0^2+1})$ is precisely the element $(\kappa(\mathfrak p_1),\rho_{x_0^2+1})$ defined in Remark \ref{explicit_c2}. It remains to compute the image of $(\kappa({\mathfrak q_2}), x_1x_3\rho_{x_2})$, but this is straightforward. 
\end{rem}

\subsection{Witt cohomology of $S^2$}
Let $S^2=\R[x_0,x_1,x_2]/(\sum x_i^2-1)$. As in \cite{DF}, we use the stereographic projection to compute the Witt cohomology of $S^2$ (and therefore its Witt groups). Explicitly, consider the algebra $\R[x_0,x_1]$ and the polynomial $f=x_0^2+x_1^2+1$. Then there is an isomorphism
$$S^2_{1-x_2}\to \R[x_0,x_1]_{f}$$
defined by $x_i\mapsto 2x_i/f$ for $i=0,1$ and $x_2\mapsto (f-2)/f$. Using the exact sequence of localization (see \cite[Theorem 1.5.5]{BalSurvey}), we get an exact sequence
$$\xymatrix@C=1.4em{0\ar[r] & H^0(S^2,W)\ar[r] & H^0(S^2_{1-x_2},W)\ar[r] & H^1_{1-x_2}(S^2,W)\ar[r] & H^1(S^2,W)\ar[r] & \ldots}$$
First we compute the Witt cohomology of $S^2_{1-x_2}$, which is the same as the Witt cohomology of $\R[x_0,x_1]_{f}$. Using the localization sequence again, we have an exact sequence:
$$\xymatrix@C=1.4em{0\ar[r] & H^0(\A^2,W)\ar[r] & H^0(\R[x_0,x_1]_{f},W)\ar[r] & H^1_f(\A^2,W)\ar[r] & H^1(\A^2,W)\ar[r] & \ldots}$$
By homotopy invariance, $H^0(\A^2,W)=\Z$ and $H^i(\A^2,W)=0$ if $i\geq 1$. Because $f$ is a regular parameter and $\R[x_0,x_1]/{f}=C^1$, we have $H^i_f(\A^2,W)\simeq H^{i-1}(C^1,W)$ by Lemma \ref{reg_param}.

The sequence and Lemma \ref{comput_c1} imply that $H^i(\R[x_0,x_1]_f,W)=0$ if $i\geq 1$. To compute $H^0(\R[x_0,x_1]_f)$, observe that the map $H^0(\R,W)\to H^0(\R[x_0,x_1]_f,W)$ admits a retraction given by the choice of any rational point. Therefore we get $H^0(\R[x_0,x_1]_f,W)\simeq \Z\oplus \Z/4$, the factor $\Z/4$ being generated by the class of the form $\langle f\rangle$. 

It remains to compute the groups $H^i_{1-x_2}(S^2,W)$. Remark that $1-x_2$ is a regular parameter in $S^2_{x_1}$ and that $S^2_{x_1}/(1-x_2)\simeq \C[x_1,x_1^{-1}]$. Again, we have an exact sequence of localization
$$\xymatrix@C=1.4em{0\ar[r] & H^1_{1-x_2}(S^2,W)\ar[r] & H^1_{1-x_2}(S^2_{x_1},W)\ar[r] & H^2_{1-x_2,x_1}(S^2,W)\ar[r] & \ldots.}$$
Now the only point of codimension $2$ with support on $(1-x_2,x_1)$ is the maximal ideal $(1-x_2,x_1,x_0)$ and then $H^2_{1-x_2,x_1}(S^2,W)\simeq \Z$. Since $1-x_2$ is a regular parameter in $S^2_{x_1}$, we get 
$$H^1_{1-x_2}(S^2_{x_1},W)\simeq H^0(S^2_{x_1}/(1-x_2),W)\simeq H^0(\C[x_1,x_1^{-1}],W)\simeq \Z/2 \oplus \Z/2,$$ 
the second factor being generated by the class of the form $\langle x_1\rangle$ (\cite[Theorem 5.6]{Gi1}). We also get $H^2_{1-x_2}(S^2_{x_1},W)=0$ and therefore $H^2_{1-x_2}(S^2,W)\simeq \Z$.

Finally, we can compute the Witt cohomology of $S^2$:

\begin{thm}\label{s2}
Let $S^2$ be the algebraic real $2$-sphere. Then 
$$H^i(S^2,W)\simeq \left\{ \begin{array}{ll} \Z\oplus \Z/2 & \mbox{if $i=0$}\\
\Z/2 & \mbox{if $i=1$}\\
\Z & \mbox{if $i=2$}\\
0 & \mbox{if $i\geq 3$}
\end{array}\right. $$
\end{thm}

\begin{proof}
We just have to examine the exact sequence
$$\xymatrix@C=1.4em{0\ar[r] & H^0(S^2,W)\ar[r] & H^0(S^2_{1-x_2},W)\ar[r] & H^1_{1-x_2}(S^2,W)\ar[r] & H^1(S^2,W)\ar[r] & \ldots}$$
In light of the above computations, the only difficulty is to understand the map $H^0(S^2_{1-x_2},W)\to H^1_{1-x_2}(S^2,W)$. But it is not hard to see that the form $\langle 1-x_2\rangle$ generating the factor $\Z/4$ in $H^0(S^2_{1-x_2},W)$ is sent to the form $\langle 1\rangle$ generating the first $\Z/2$ factor in $H^1_{1-x_2}(S^2,W)$.
\end{proof}

\begin{rem}\label{explicit_s2}
Let $\mathfrak q$ be the prime ideal of $S^2$ generated by $(1-x_2)$. The Koszul complex $Kos(1-x_2)$ associated to $(1-x_2)$ can be taken as a generator of $\exten 1{S^2_{\mathfrak q}}{\kappa(\mathfrak q)}{S^2_{\mathfrak q}}$ and we get a symmetric isomorphism
$$\rho_{1-x_2}:\kappa(\mathfrak q)\to \exten 1{S^2_{\mathfrak q}}{\kappa(\mathfrak q)}{S^2_{\mathfrak q}}$$
defined by $\rho_{1-x_2}(1)=Kos(1-x_2)$. Our computations above proves that a generator of $H^1(S^2,W)$ is given by the class of the cycle $(\kappa(\mathfrak q),x_1\rho_{1-x_2})$. 
\end{rem}

\subsection{The computation of $W^2(S^3)$}

We finally compute $W^2(S^3)\simeq H^2(S^3,W)$. To achieve this, we use our computations of the previous sections.

\begin{thm}\label{s3prime}
Let $S^3$ be the algebraic real $3$-sphere. Then $W^2(S^3)=0$.
\end{thm}

\begin{proof}
We use the localization sequence
$$\xymatrix@C=1.2em{\ldots\ar[r] & H^1(S^3,W)\ar[r] & H^1(S^3_{x_3},W)\ar[r] & H^2_{x_3}(S^3,W)\ar[r] & H^2(S^3,W)\ar[r] & H^2(S^3_{x_3},W).}$$
Since $S^3_{x_3}\simeq B^3_{x_2+x_3}$, we get $H^2(S^3_{x_3},W)=0$ by Lemma \ref{comput_b3}. Using Lemma \ref{reg_param}, we get an isomorphism $H^2_{x_3}(S^3,W)\simeq H^1(S^2,W)$. The above exact sequence reduces then to an exact sequence
$$\xymatrix@C=1.2em{\ldots\ar[r] & H^1(S^3,W)\ar[r] & H^1(S^3_{x_3},W)\ar[r] & H^1(S^2,W)\ar[r] & H^2(S^3,W)\ar[r] & 0.}$$
It suffices therefore to prove that the homomorphism $H^1(S^3_{x_3},W)\to H^1(S^2,W)$ is surjective to conclude. Using the isomorphism $\phi:B^3_{x_2+x_3}\to S^3_{x_3}$ and the explicit generator of Remark \ref{explicit_b3}, we get an explicit generator of $H^1(S^3_{x_3},W)$ as follows:

We get $\phi(\mathfrak q_1)=(x_0^2+x_3^2)=:\mathfrak q_1^\prime$ and 
$$\phi(Kos(x_0^2+1))=Kos(x_0^2/x_3^2+1)=(1/x_3^2)Kos(x_0^2+x_3^2).$$
Moreover, $\phi(x_0x_1+x_2)=(1/x_3^2)(x_0x_1+(1-x_2)x_3)$. Letting $\rho_{x_0^2+x_3^2}$ be the usual symmetric isomorphism, we see that the image of $(\kappa(\mathfrak q_1),(x_0x_1+x_2)\rho_{x_0^2+1})$ under $\phi$ is $(\kappa(\mathfrak q_1^\prime),(x_0x_1+(1-x_2)x_3)\rho_{x_0^2+x_3^2})$ (since we work in a Witt group, we can forget the factors $(1/x_3^2)$).

Similarly, we get $\phi(\mathfrak q_2)=(1-x_2)=:\mathfrak q_2^\prime$, $\phi(Kos(x_2))=(1/x_3)Kos(1-x_2)$ and $\phi(x_1x_3)=(1/x_3^2)x_1(1+x_2)$. Hence the image of $(\kappa(\mathfrak q_2),x_1x_3\rho_{x_2})$ under $\phi$ is $(\kappa(\mathfrak q_2^\prime),x_1x_3\rho_{1-x_2})$ since $1+x_2\equiv 2$ modulo $\mathfrak q_2^\prime$ and $2$ is a square.

We are therefore reduced to compute the image of 
$$(\kappa(\mathfrak q_1^\prime),(x_0x_1+(1-x_2)x_3)\rho_{x_0^2+x_3^2})+ (\kappa(\mathfrak q_2^\prime),x_1x_3\rho_{1-x_2})$$ 
under the homomorphism $H^1(S^3_{x_3},W)\to H^1(S^2,W)$. A direct computation shows that the image of $(\kappa(\mathfrak q_1^\prime),(x_0x_1+(1-x_2)x_3)\rho_{x_0^2+x_3^2})$ is trivial, while the image of $(\kappa(\mathfrak q_2^\prime),x_1x_3\rho_{1-x_2})$ is the cycle $(\kappa(\mathfrak q),x_1\rho_{1-x_2})$ of Remark \ref{explicit_s2}. Hence the homomorphism $H^1(S^3_{x_3},W)\to H^1(S^2,W)$ is surjective and the result is proved.
\end{proof}

\begin{cor}\label{s3}
Let $S^3$ be the algebraic real $3$-sphere. Then $\K {S^3}=0$.
\end{cor}

\begin{proof}
It is well-known that $K_0(S^3)=\Z$ (using the fact that in this case the algebraic $K$-theory and the topological $K$-theory coincide by \cite[Theorem 3]{Sw2}), and therefore we get $\K {S^3}=W^2(S^3)$ by definition.
\end{proof}

\subsection{Projective modules on $S^3$}

Now we have all the tools in hand to prove the following theorem:

\begin{thm}\label{main}
Every finitely generated projective module over the real algebraic $3$-sphere is free. 
\end{thm}

\begin{proof}
Because $K_0(S^3)=0$, all the projective modules are stably free. We know already that those of rank $3$ are free by \cite[Proposition 5.11]{Fa4}. In view of Theorem \ref{cancellation} and our computation of $\K {S^3}$, it remains to show that $Um_4(S^3)/Sp_4(S^3)$ is trivial.
By \cite[Theorem 4.9, Theorem 5.6]{Fa4}, the group $Um_4(S^3)/E_4(S^3)$ is generated by the class of the unimodular row $(x_0,x_1,x_2,x_3)$. Now the matrix
$$M:=\begin{pmatrix} x_0 & x_1 & x_2 & x_3 \\
-x_1 & x_0 & -x_3 & x_2 \\
-x_2 & x_3 & x_0 & -x_1 \\
-x_3 & -x_2 & x_1 & x_0\end{pmatrix}$$
is in $\Sp {S^3}\subset SL_4(S^3)$. Mapping a matrix (of determinant $1$) to its first row yields a homomorphism of abelian groups 
$$r:SL_4(S^3)/E_4(S^3)\to Um_4(S^3)/E_4(S^3)$$
by \cite[Theorem 5.3(ii)]{vdk}. It follows that for any $[v]$ in $Um_4(S^3)/E_4(S^3)$ there exists $n\in \Z$ such that $[v]=r(M^n)$. This yields $[v]+r(M^{-n})=0$ and then $vM^{-n}=e_1E$ for some $E\in E_4(S^3)$. Now $e_1E_4(S^3)=e_1ESp_4(S^3)$ by Vaserstein's theorem (whose french translation can be found in \cite[Proposition 6']{Bass3}). This proves that $v\in e_1Sp_4(S^3)$. Whence the result. 
\end{proof}


\section{Acknowledgements}
It is a pleasure to thank Manuel Ojanguren for stimulating my interest in the subject. I also wish to express my gratitude to the referee for his comments and corrections.


\bibliography{biblio_cancellation.bib}{}

\begin{thebibliography}{10}

\bibitem{BalSurvey}
Paul Balmer.
\newblock Witt groups.
\newblock In {\em Handbook of $K$-theory. Vol. 1, 2}, pages 539--576. Springer,
  Berlin, 2005.

\bibitem{BW}
Paul Balmer and Charles Walter.
\newblock A {G}ersten-{W}itt spectral sequence for regular schemes.
\newblock {\em Ann. Sci. \'Ecole Norm. Sup. (4)}, 35(1):127--152, 2002.

\bibitem{BO}
Jean Barge and Manuel Ojanguren.
\newblock Fibr\'es alg\'ebriques sur une surface r\'eelle.
\newblock {\em Comment. Math. Helv.}, 62(4):616--629, 1987.

\bibitem{Bass2}
Hyman Bass.
\newblock Unitary algebraic {$K$}-theory.
\newblock In {\em Algebraic {K}-theory, {III}: {H}ermitian {K}-theory and
  geometric applications ({P}roc. {C}onf., {B}attelle {M}emorial {I}nst.,
  {S}eattle, {W}ash., 1972)}, pages 57--265. Lecture Notes in Math., Vol. 343.
  Springer, Berlin, 1973.

\bibitem{Bass3}
Hyman Bass.
\newblock Lib\'eration des modules projectifs sur certains anneaux de
  polyn\^omes.
\newblock In {\em S\'eminaire {B}ourbaki, 26e ann\'ee (1973/1974), {E}xp. {N}o.
  448}, pages 228--354. Lecture Notes in Math., Vol. 431. Springer, Berlin,
  1975.

\bibitem{CS}
J.-L. Colliot-Th{\'e}l{\`e}ne and C.~Scheiderer.
\newblock Zero-cycles and cohomology on real algebraic varieties.
\newblock {\em Topology}, 35(2):533--559, 1996.

\bibitem{DF}
Ivo Dell'Ambrogio and Jean Fasel.
\newblock The {W}itt groups of the spheres away from two.
\newblock {\em J. Pure Appl. Algebra}, 212(5):1039--1045, 2008.

\bibitem{EL}
Richard Elman and T.~Y. Lam.
\newblock Classification theorems for quadratic forms over fields.
\newblock {\em Comment. Math. Helv.}, 49:373--381, 1974.

\bibitem{FS}
J.~Fasel and V.~Srinivas.
\newblock Chow-{W}itt groups and {G}rothendieck-{W}itt groups of regular
  schemes.
\newblock {\em Adv. Math.}, 221(3):302--329, 2009.

\bibitem{Fa1}
Jean Fasel.
\newblock Groupes de {C}how-{W}itt.
\newblock {\em M\'em. Soc. Math. Fr. (N.S.)}, (113):viii+197, 2008.

\bibitem{Fa4}
Jean Fasel.
\newblock Some remarks on orbit sets of unimodular rows.
\newblock Preprint available at
  \texttt{http://www.math.uiuc.edu/K-theory/0905/}, to appear in Comment. Math.
  Helv., 2008.

\bibitem{Fa2}
Jean Fasel.
\newblock Stably free modules over smooth affine threefolds.
\newblock arXiv:0911.3495, to appear in Duke Math. Journal, 2009.

\bibitem{Fo}
Robert~M. Fossum.
\newblock Vector bundles over spheres are algebraic.
\newblock {\em Invent. Math.}, 8:222--225, 1969.

\bibitem{Gi1}
Stefan Gille.
\newblock On {W}itt groups with support.
\newblock {\em Math. Ann.}, 322(1):103--137, 2002.

\bibitem{Gi2}
Stefan Gille.
\newblock A graded {G}ersten-{W}itt complex for schemes with a dualizing
  complex and the {C}how group.
\newblock {\em J. Pure Appl. Algebra}, 208(2):391--419, 2007.

\bibitem{Mon}
Jean-Philippe Monnier.
\newblock Witt group and torsion {P}icard group of real curves.
\newblock {\em J. Pure Appl. Algebra}, 169(2-3):267--293, 2002.

\bibitem{OVV}
D.~Orlov, A.~Vishik, and V.~Voevodsky.
\newblock An exact sequence for {$K\sp M\sb \ast/2$} with applications to
  quadratic forms.
\newblock {\em Ann. of Math. (2)}, 165(1):1--13, 2007.

\bibitem{Suj}
R.~Sujatha.
\newblock Witt groups of real projective surfaces.
\newblock {\em Math. Ann.}, 288(1):89--101, 1990.

\bibitem{Sw2}
Richard~G. Swan.
\newblock {$K$}-theory of quadric hypersurfaces.
\newblock {\em Ann. of Math. (2)}, 122(1):113--153, 1985.

\bibitem{Sw1}
Richard~G. Swan.
\newblock Vector bundles, projective modules and the {$K$}-theory of spheres.
\newblock In {\em Algebraic topology and algebraic {$K$}-theory ({P}rinceton,
  {N}.{J}., 1983)}, volume 113 of {\em Ann. of Math. Stud.}, pages 432--522.
  Princeton Univ. Press, Princeton, NJ, 1987.

\bibitem{vdk}
Wilberd van~der Kallen.
\newblock A module structure on certain orbit sets of unimodular rows.
\newblock {\em J. Pure Appl. Algebra}, 57(3):281--316, 1989.

\bibitem{Vo}
Vladimir Voevodsky.
\newblock Motivic cohomology with {${\bf Z}/2$}-coefficients.
\newblock {\em Publ. Math. Inst. Hautes \'Etudes Sci.}, (98):59--104, 2003.

\end{thebibliography}
\bibliographystyle{plain}


\end{document}